\documentclass[10pt,leqno]{amsart} %%%%leqno es para numeraci\'{o}n de
   \usepackage[centertags]{amsmath}
   \usepackage{amsfonts}
   \usepackage{amsmath}
   \usepackage{amssymb}
   \usepackage{amsthm}
      \usepackage{color}
   \usepackage{newlfont}
   \usepackage[latin1]{inputenc}%%%%para que reconozca las tildes.
\allowdisplaybreaks[3]
\theoremstyle{plain}
\newtheorem{theorem}{Theorem}[section]
\newtheorem{lemma}[theorem]{Lemma}

\theoremstyle{definition}

\theoremstyle{remark}
\newtheorem{remark}[theorem]{Remark}
\newtheorem*{remark*}{Remark}

\numberwithin{equation}{section}

\newcommand\D{{\mathcal D}}
\newcommand\A{{\mathcal A}}

\newcommand\CC{{\mathbb C}}

\newcommand\RR{{\mathbb R}}

\newcommand\PP{{\mathbb P}}

\title[$\D$-operators, orthogonal polynomials and $q$-difference equations]
{Using $\D$-operators to construct orthogonal polynomials satisfying  higher order $q$-difference equations}

\author{Renato Álvarez-Nodarse and Antonio J. Dur\'an}
\address{A. J. Dur\'an \\
Departamento de An\'{a}lisis Matem\'{a}tico \\
Universidad de Sevilla \\
Apdo (P. O. BOX) 1160\\
41080 Sevilla. Spain.}
\email{duran@us.es}
\address{R. Álvarez-Nodarse \\
Departamento de An\'{a}lisis Matem\'{a}tico \\
Universidad de Sevilla \\
Apdo (P. O. BOX) 1160\\
41080 Sevilla. Spain.}
\email{ran@us.es}
\thanks{Partially supported by MTM2012-36732-C03-03 (Ministerio de Economía y Competitividad),
FQM-262, FQM-4643, FQM-7276 (Junta de Andalucía) and Feder Funds (European
Union).}



   \date{}
   
\begin{document}
   \maketitle

   \begin{abstract}
  Let  $(p_n)_n$ be either the $q$-Meixner or the $q$-Laguerre polynomials. We form a new sequence of polynomials $(q_n)_n$ by considering a linear combination of two consecutive $p_n$: $q_n=p_n+\beta_np_{n-1}$, $\beta_n\in \RR$. Using the concept of $\D$-operator, we generate sequences $(\beta _n)_n$ for which the polynomials $(q_n)_n$ are orthogonal with respect to a measure and common eigenfunctions of a higher order $q$-difference operator.
\end{abstract}

\section{Introduction}
The most important families of orthogonal polynomials are the classical, classical discrete or $q$-classical families (Askey scheme and its $q$-analogue). Besides the orthogonality, they are also common eigenfunctions of a second order differential, difference or $q$-difference operator, respectively.

The  issue of orthogonal polynomials which are also common eigenfunctions of a higher order differential operator was raised by H.L. Krall in 1939, when he obtained a complete classification for the case of a differential operator of order four (\cite{Kr2}). After his pioneer work, orthogonal polynomials which are also common eigenfunctions of higher order differential operators are usually called Krall polynomials.  Since the eighties a lot of effort has been devoted to find Krall polynomials. Roughly speaking, one can construct Krall polynomials by using the Laguerre $x^\alpha e^{-x}$, or Jacobi weights $(1-x)^\alpha(1+x)^\beta$, assuming that one or two of the parameters $\alpha$ and $\beta$ are nonnegative integers and adding a linear combination of Dirac deltas and their derivatives at the endpoints of the orthogonality interval (\cite{Kr2}, \cite{koekoe}, \cite{koe} \cite{koekoe2}, \cite{L1}, \cite{L2}, \cite{GrH1}, \cite{GrHH}, \cite{GrY}, \cite{Plamen1}, \cite{Plamen2}, \cite{Zh}).

Discrete versions of Krall problem appeared at the beginning of the nineties. Richard Askey posed in 1991 (see page 418 of \cite{BGR}) the problem of finding orthogonal polynomials which are also common eigenfunctions of a higher order difference operator (discrete Krall polynomials), and suggested that such examples could be found by adding a Dirac delta at $0$ to the Meixner measure. Bavinck, van Haeringen and Koekoek answered in the negative Askey's question (see  \cite{BH} and \cite{BK}). Actually, the first examples of discrete Krall polynomials have appeared very recently: it has been introduced by one of us in 2012 (see \cite{du1}). Orthogonalizing measures for these families of polynomials are generated by multiplying the classical discrete weights of Charlier, Meixner, Krawtchouk and Hahn by certain variants of the annihilator polynomial of a set of numbers. This kind of transformation which consists in multiplying a measure times a polynomial is called a Christoffel transform.

As far as the authors know the first examples of $q$-Krall polynomials were introduced by Grünbaum and Haine in 1996 \cite{GrH2}.
By applying Darboux transform to the little $q$-Jacobi polynomials $p_n^{a,b;q}$, with $a=q$, they produced polynomials which are eigenfunctions of a fourth order $q$-difference operator. Those polynomials are also orthogonal with respect to the little $q$-Jacobi measure (with parameter $a=1$) plus a Dirac delta at $0$. They obtained a similar result for the $q$-Laguerre polynomials $p_n^{\alpha ;q}$ with $\alpha =1$. Using a different approach, Vinet and Zhedanov \cite{VZ} generalized Grünbaum-Haine results for the little $q$-Jacobi polynomials with parameter $a=q^j$, and for the little $q$-Laguerre polynomials also with parameter $a =q^j$, where $j$ is a positive integer in both cases. We remark that in those examples, the orthogonality measures for the new polynomials follow the same pattern that for the  Krall polynomials.
Something similar was obtained by Haine and Iliev \cite{HI} for the Askey-Wilson polynomials $(p_n^{a,b,c,d;q})_n$:  assuming that $a=\pm q^{\alpha +l/2}$, $d=q^{l/2}$, with $\alpha , l $ positive integers, they show that by adding Dirac deltas (with arbitrary weights) at points of the form $\pm (q^{k+l/2}+q^{-(k+l/2)})$, $0\le k\le \alpha -1$,
to the Askey-Wilson measure, orthogonal polynomials can be generated which are also eigenfunctions of $q$-difference operators of arbitrary orders.

In the paper  \cite{DD} one of us has introduced the concept of $\D$-operator associated to a sequence of polynomials
$(p_n)_n$ and an algebra $\A$ of operators acting in the linear space of polynomials (see Section 3 below). In that paper,  $\D$-operators have been used to provide a method which generates Krall and discrete Krall  polynomials.
The purpose of this paper is to show that $\D$-operators are also a very useful tool to generate $q$-Krall  polynomials.

To do that, we consider in this paper $q$-Meixner and $q$-Laguerre polynomials, and the algebra $\A _q$ formed by all finite order $q$-difference
operators $D:\PP \to \PP $ of the form
\begin{equation}\label{algdiff}
D(p)=\sum_{j=s}^rf_j(x)p(xq^{j}),\quad \mbox{$s$ and $r$ integers with $s\le r$},
\end{equation}
where $f_j$, $j=s,\cdots r$, are rational functions (the order of $D$ is then $r-s$). Denote both $q$-classical families of polynomials by $(p_n)_n$. Using $\D$-operators, we produce sequences of numbers $(\beta _n)_n$ such that the polynomials $q_n(x)=p_n(x)+\beta_np_{n-1}(x)$, ($p_{-1}=0$), are orthogonal with respect to a measure and common eigenfunctions of a higher order $q$-difference operator $D\in \A _q$.

For $q$-Meixner polynomials we have found three different $\D$-operators, and then we produce such sequences $(\beta _n)_n$ in three different ways (Section 4). In the three cases, the orthogonality measures for the polynomials $(q_n)_n$ turn out to be the product of certain polynomials times the $q$-Meixner measure (that is, Christoffel transform of the $q$-Meixner measure). To make this introduction more useful to the reader, let us provide one of our results in more details. Denote by $(m_{n}^{b,c;q})_n$ the sequence of $q$-Meixner
polynomials (see (\ref{defmep}) below). For $0<q<1$, $0\le bq<1$ and $c>0$, the $q$-Meixner polynomials are
orthogonal with respect to the positive measure denoted by $\rho _{b,c}^q$ (see (\ref{mew}) below). Fixed a positive integer $k$ and assume that $m_{k}^{-c,1/(bc);q}(q^{n})\not =0$, $n\ge 1$. Define now the sequence of numbers $(\beta _n)_{n\ge 1}$  by
$$
\beta_n=\frac{m_{k}^{-c,1/(bc);q}(q^{n+1})}{m_{k}^{-c,1/(bc);q}(q^{n})},\quad n\ge 1.
$$
Notice that we have a sequence $(\beta _n)_n$ for each $k\ge 1$. Then, the polynomials $q_n(x)=m_{n}^{b,c;q}(x)+\beta _nm_{n-1}^{b,c;q}(x)$,  ($m_{-1}^{b,c;q}=0$), are orthogonal with respect to the positive measure
$$
(x+bcq) \cdots ( x+bcq^{k}) \rho_{b,q^{k+1}c}^q.
$$
Moreover, they are also eigenfunctions of a $q$-difference operator of
order $2k+2$. Notice that neither the parameter $b$ nor the parameter $c$ need to be a positive integer.

This way of producing $q$-Krall orthogonal polynomials  is different from what one can find in the literature (\cite{GrH2},\cite{VZ},\cite{HI}).
As we have mentioned above,  the orthogonality measures for the extant $q$-Krall polynomials are formed by adding Dirac deltas to some particular instances of
$q$-classical measures. In the case of the $q$-Meixner polynomials, the pattern we have found in this paper is to multiply the $q$-Meixner weight by certain polynomials. Hence, it is  similar to what happens in the discrete case.
In fact, the results in Section 4 of this paper give rise to the following conjecture:
\medskip

\noindent
\textbf{Conjecture A.} Let $F_1$, $F_2$ and $F_3$ be three finite sets of positive integers (the empty set is allowed), then the orthogonal polynomials with respect to the  measure (if there exist)
$$
\tilde \rho_{F_1,F_2,F_3}^{b,c;q}=\left(\prod_{f\in F_1}(x+bc/q^f)\right) \left(\prod_{f\in F_2}(x-bq^ {f+1})\right)\left(\prod_{f\in F_3}(x-1/q^f)\right)\rho_{b,c}^q
$$
are eigenfunctions of a $q$-difference operator of order
$$
\sum_{i=1}^3\left( 2\sum_{f\in F_i}f-n_{F_i}(n_{F_i}-1)\right) +2
$$
where $n_{F_i}$ is the number of elements of $F_i$.

That sequence of orthogonal polynomials always exists when the measure $\tilde \rho_{F_1,F_2,F_3}^{b,c;q}$ is positive, which happens for any $F_1$, $F_2$ and, for instance, $F_3=\{k,k+1,\cdots , k+2m-1\}$.
\bigskip

For the $q$-Laguerre polynomials $(L_n^{\alpha ;q})_n$ we have found two $\D$-operators (Section 5). Using them, we show that from $q$-Laguerre polynomials one can generate $q$-Krall polynomials using both procedures discussed above, that is, multiplying the $q$-Laguerre weight by certain polynomials (no constraint then on the parameter $\alpha$) or adding to the $q$-Laguerre weight a Dirac delta at $0$ (assuming that $\alpha$ is a positive integer). In this last case, we recover the results by Vinet-Zhedanov \cite{VZ}.

\section{Preliminaries}
Along this paper, $q$ denotes a real number $q\not =\pm 1$. The $q$-derivative $D_q$ is then defined as:
\begin{equation}\label{defdq}
D_qf=\frac{f(qx)-f(x)}{x(q-1)}.
\end{equation}

For a linear operator $D:\PP \to \PP$ and a polynomial $P(x)=\sum _{j=0}^ka_jx^j$, the operator $P(D)$ is defined in the usual way
$P(D)=\sum _{j=0}^ka_jD^j$.

As usual $(a;q)_j$ will denote the $q$-Pochhammer symbol defined by
\begin{align*}
(a;q)_0=1,\quad \quad (a;q)_j&=(1-a)(1-aq)\cdots (1-aq^{j-1}),\quad \mbox{for $j\ge 1$, $a\in \CC$.}\\
(a;q)_\infty&=\prod_{j=0}^\infty(1-aq^j).
\end{align*}
We also define
$$
(a,\cdots , b;q)_j=(a;q)_j\cdots (b;q)_j.
$$

Let $\mu$ be a moment functional on the real line, that is, a linear mapping $\mu :\PP \to \RR$.
The $n$-th moment of $\mu $ is defined by $\mu_n=\langle \mu, x^n\rangle $.
It is well-known that any moment functional on the real line can be represented by integrating with respect to a Borel measure
(positive or not) on the real line
(this representation is not unique \cite{du0}).
If we also denote this measure by $\mu$, we have $\langle \mu,p\rangle=\int p(x)d\mu(x)$ for all polynomial $p\in \PP$. Taking this into account,
we will conveniently use along this paper one or other terminology (orthogonality with respect to a moment functional or with
respect to a measure). We say that a sequence of polynomials $(p_n)_n$, $p_n$ of degree $n$, $n\ge 0$, is orthogonal with respect to the moment functional $\mu$ if $\langle \mu, p_np_m\rangle=0$, for $n\not =m$.

Favard's Theorem establishes that a sequence $(p_n)_n$ of polynomials, $p_n$ of degree $n$, is orthogonal with respect to a moment functional if
and only if it satisfies
a three term recurrence relation of the form ($p_{-1}=0$)
$$
xp_n(x)=a_np_{n+1}(x)+b_np_n(x)+c_np_{n-1}(x), \quad n\ge 0,
$$
where $(a_n)_n$, $(b_n)_n$ and $(c_n)_n$ are sequences of real numbers with
$a_{n-1}c_n\not =0$, $n\ge 1$. If, in addition, $a_{n-1}c_n>0$, $n\ge 1$,
then the polynomials $(p_n)_n$ are orthogonal with respect to a moment functional which can be represented by a positive measure, and conversely.

\bigskip
The kind of transformation which consists in multiplying a moment functional $\mu$ times a polynomial $r$ is called a Christoffel transform. The new
moment functional $r\mu$ is defined by $\langle r\mu,p\rangle =\langle \mu,rp\rangle $.
It has a long tradition in the context of orthogonal polynomials: it goes back a century and a half ago when
E.B. Christoffel (see \cite{Chr} and also \cite{Sz}) studied it for the particular case $r(x)=x$.

For a real number $\lambda$, the moment functional $\mu (x+\lambda)$ is defined in the usual way
$\langle \mu (x+\lambda),p\rangle =\langle \mu ,p(x-\lambda)\rangle $. Hence, if $(p_n)_n$ are orthogonal polynomials with respect
to $\mu$ then $(p_n(x+\lambda))_n$ are orthogonal with respect to $\mu(x+\lambda)$.

We will use the following straightforward Lemma (Lemma 2.1 of \cite{DD}) to construct orthogonal
polynomials with respect to a measure together with a Dirac delta (for related results see \cite{GPV}, \cite{Y}).

\begin{lemma}\label{addel}
From a measure $\nu$ and a real number $\lambda$, we define the measure $\mu$ by $\mu=(x-\lambda)\nu$. Assume that we have a sequence $(p_n)_n$ of
orthogonal polynomials with respect to $\mu$ and write $\alpha_n=\int p_nd\nu$.
For a given real number $M$ we write $\tilde \rho =\nu + M\delta_\lambda$ and define the numbers
\begin{equation}\label{hk2}
\beta_n=-\frac{\alpha_{n}+Mp_n(\lambda)}{\alpha_{n-1}+Mp_{n-1}(\lambda)}, \quad n\ge 1
\end{equation}
(we implicitly assume that $\alpha_{n-1}+Mp_{n-1}(\lambda)\not =0$, $n\ge 1$).
Then the polynomials defined by $q_0=1$ and $q_n=p_n+\beta_np_{n-1}$, are orthogonal with respect to $\tilde \rho$.
\end{lemma}

\section{$\D $-operators}
In this Section, we recall the basic facts about $\D$-operators (for more details see \cite{DD}).

The starting point is a sequence of polynomials $(p_n)_n$, $p_n$ of degree $n$, and an algebra of operators $\A $ acting in the linear space of polynomials.
In addition, we assume that the polynomials $p_n$, $n\ge 0$, are eigenfunctions of certain operator $D^P\in \A$. We write $(\theta_n)_n$ for the corresponding eigenvalues,
so that $D^P(p_n)=\theta_np_n$, $n\ge 0$.

Given a sequence of numbers $(\beta _n)_n$, we define a new sequence of polynomials $(q_n)_n$ by $q_0=1$ and
\begin{equation}\label{qncc}
q_n=p_n+\beta_np_{n-1}, \quad n\ge 1.
\end{equation}
Assuming that the sequence $(\beta _n)_n$ has certain appealing form, our procedure provides a new operator $D^Q\in \A$ for which the polynomials $(q_n)_n$ are eigenfunctions.
That appealing form for the sequence $(\beta _n)_n$ will be determined by the concept of $\D$-operator.

A $\D$-operator associated to the algebra $\A$ and the sequence of polynomials $(p_n)_n$ is defined from the two sequences of numbers $(\varepsilon_n)_n$ and $(\sigma_n)_n$ as follows.
We consider the operator $\D$, $\D :\PP \to \PP $, defined by linearity from
\begin{equation}\label{defTo2}
\D (p_n)=-\frac{1}{2}\sigma_{n+1} p_n(x)+\sum _{j=1}^n (-1)^{j+1}\sigma_{n+1-j}\left(\prod_{i=1}^j\varepsilon _{n-i+1}\right) p_n(x) .
\end{equation}
We then say that $\D$ is a  $\D$-operator if $\D\in \A$.

The appealing form mentioned above for the sequence $(\beta_n)_n$ (which defines the polynomials $(q_n)_n$ (\ref{qncc})) is then
$\displaystyle \beta_n =\varepsilon _n\frac{P(\theta_n)}{P(\theta_{n-1})}$, where $P$ is an arbitrary polynomial.

The method which we will use to find $q$-Krall polynomials is included in the following lemma.

\begin{lemma}\label{fl2v} (Lemma 3.2 of \cite{DD})
Let $\A$ and $(p_n)_n$ be, respectively, an algebra of operators acting in the linear space of polynomials, and a sequence of polynomials, $p_n$ of degree $n$.
We assume that $(p_n)_n$ are eigenfunctions of an operator $D^P\in \A$, that is, there exist  numbers $\theta _n$, $n\ge 0$, such that
$D^P(p_n)=\theta_np_n$, $n\ge 0$. We also have two sequences of numbers $(\varepsilon_n )_n$ and $(\sigma_n)_n$ which define a $\D$-operator  for $(p_n)_n$ and $\A$
(see (\ref{defTo2})).
For an arbitrary polynomial $P_2$ such that $P_2(\theta_n)\not =0$, $n\ge 0$, we define the sequences of numbers $(\gamma_{n\ge 1})_n$ and $(\lambda_n)_{n\ge 1}$ by
\begin{align}\label{gamp}
\gamma_{n+1}&=P_2(\theta_n),\quad n\ge 0, \\\label{lamp}
\lambda_n-\lambda_{n-1}&=\sigma_n\gamma _n, \quad n\ge 1,
\end{align}
and assume that there exists a polynomial $P_1$ such that $\lambda_{n+1}+\lambda_n=P_1(\theta_n)$, $n\ge 0$.
We finally define the sequence of polynomials $(q_n)_n$ by $q_0=1$ and
\begin{equation}\label{defqng2}
q_n=p_n+\beta_np_{n-1},\quad n\ge 1,
\end{equation}
where the numbers $\beta_n$, $n\ge 1$, are given by
\begin{equation}\label{defbetng2}
\beta_n=\varepsilon_n \frac{\gamma_{n+1}}{\gamma_n}.
\end{equation}
Then $D^Q(q_n)=\lambda _nq_n$ where the operator $D^Q$ is defined by
\begin{equation}\label{defD2}
D^Q=\frac{1}{2}P_1(D^P)+\D P_2(D^P).
\end{equation}
Moreover $D^Q\in \A$.
\end{lemma}

\begin{remark}\label{rem}
For the examples considered in this paper, we have  $\theta _n=uq^n$ and $\sigma_n=vq^n$, where $u$ and $v$ are constants
independent of $n$. In this case, we can easily see that whatever the polynomial $P_2$ is, one can find a polynomial $P_1$ satisfying the hypothesis of the previous Lemma. Indeed, if $P_2(x)=\sum_{j=0}^kw_jx^j$ take
\begin{equation}\label{defP1}
P_1(x)=\frac{vqx}{u}\left(P_2(x)-2\sum_{j=0}^k\frac{w_j}{1-q^{j+1}}x^j\right).
\end{equation}
From the definition of $P_1$, it straightforwardly follows that
\begin{equation}\label{pppp}
D_q(P_1(x))=\frac{vq}{u(q-1)}(P_2(x)+qP_2(qx)).
\end{equation}
For $\lambda_0=(P_1(\theta _0)-\sigma_1P_2(\theta _0))/2$, consider the sequences $(\gamma_{n})_{n\ge 1}$ and
$(\lambda_n)_{n\ge 1}$ defined by (\ref{gamp}) and (\ref{lamp}), respectively.
We have then to check that
\begin{equation}\label{ppp}
\lambda_n+\lambda_{n+1}=P_1(uq^{n}), \quad n\ge 1.
\end{equation}
One can easily see that for $n\ge 1$ (\ref{lamp}) and (\ref{ppp}) are equivalent to
$$
P_1(uq^{n+1})-P_1(uq^{n})=vq^{n+1} P_2(uq^{n})+vq^{n+2}P_2(uq^{n+1}),
$$
which it straightforwardly follows from (\ref{pppp}) by putting $x=uq^n$.
\end{remark}

\section{$q$-Meixner case}\label{sec-mei}
In this section we will show how the method works for the $q$-Meixner polynomials.
We start by recalling the basic definitions and facts about this family of $q$-polynomials.

Let $b,c$ be two real numbers satisfying that $b\not =q^{-n}$, $n\ge 1$, and $c\not =0$.
We define the sequence of $q$-Meixner polynomials $(m_{n}^{b,c;q})_n$
%we write $(m_{n}^{b,c;q})_n$ for the sequence of Meixner polynomials defined
by (recall that $q\not =\pm 1$)
\begin{equation}\label{defmep}
m_{n}^{b,c;q}(x)=\frac{(-1)^n}{(q;q)_n}\sum _{j=0}^n \frac{(q^{-n};q)_j(x;q)_j}{(bq;q)_j(q;q)_j}\left(-\frac{q^{n+1}}{c}\right)
^{j}
\end{equation}
(we have taken a slightly different normalization from the one used in  \cite{KLS}, pp, 488-492, from where
the next formulas can be easily derived).
The $q$-Meixner polynomials are eigenfunctions of the following second order $q$-difference operator
\begin{align}\label{sodeme}
D_{b,c}(p)&=\frac{c(x-bq)}{x^2}p(x/q)-\left( \frac{c(x-bq)+(x-1)(x+bc)}{x^2}-1\right) p(x)\\\nonumber &\hspace{1cm}+\frac{(x-1)(x+bc)}{x^2}p(qx),
\qquad D_{b,c} (m_{k}^{b,c;q})=q^km_{k}^{b,c;q},\quad k\ge 0.
\end{align}
They satisfy the following three term recurrence formula ($m_{-1}^{b,c;q}=0$)
\begin{equation}\label{trme}
xm_n^{b,c;q}(x)=a_nm_{n+1}^{b,c;q}(x)+b_nm_n^{b,c;q}(x)+c_nm_{n-1}^{b,c;q}(x), \quad n\ge 0
\end{equation}
where
\begin{align}\label{ttrrm}
a_n&=\frac{c(1-q^{n+1})(1-bq^{n+1})}{q^{2n+1}},\qquad
c_n=\frac{c+q^{n}}{q^{2n}},
\\ \nonumber
b_n&=1+\frac{c(1-bq^{n+1})}{q^{2n+1}}+\frac{(1-q^n)(c+q^{n})}{q^{2n}}.
\end{align}
%(to simplify the notation we remove the parameters in some formulas).
Hence, for $b\not =q^{-n}$ and $c\not =0, -q^n$, $n\ge 1$,
they are always orthogonal with respect to a moment functional $\rho_{b,c}^q$, which we
normalize
by taking $\langle \rho_{b,c}^q,1\rangle =1$. For $0<q<1$, $0\le bq<1$ and $c>0$,
this moment functional can be represented by the positive measure
\begin{equation}\label{mew}
\rho_{b,c}^q=\frac{(-bcq;q)_\infty }{(-c;q)_\infty }\sum _{x=0}^\infty \frac{(bq;q)_xc^xq^{\binom{x}{2}}}{(q,-bcq;q)_x}\delta _{q^{-x}}.
\end{equation}

For the $q$-Meixner polynomials we have found three different $\D$-operators which  have been included in the following lemma.

\begin{lemma}\label{lTme}
For $b\not =q^{-n}$ and $c\not =0, -q^n$, $n\ge 1$, consider the $q$-Meixner polynomials
$(m_{n}^{b,c;q})_n$ defined in (\ref{defmep}). Then, the operators $\D _i$,
$i=1,2,3$, defined by (\ref{defTo2}) from the sequences ($n\ge 0$)
\begin{align}\label{veme1}
\varepsilon _{n,1}&=1,\quad &\sigma_{n,1}&=\frac{q^{n-1}}{(q-1)},\\\label{veme2}
\varepsilon _{n,2}&=\frac{1}{1-bq^n},&\sigma_{n,2}&=\frac{q^n}{c(1-q)},\\\label{veme3}
\varepsilon _{n,3}&=\frac{c+q^n}{c(1-bq^n)}, &\sigma_{n,3}&=\frac{q^{n-1}}{(q-1)},
\end{align}
are $\D$-operators for $(m_{n}^{b,c;q})_n$ and the algebra $\A _q$  (\ref{algdiff}) of $q$-difference
operators  with rational coefficients.
More precisely
\begin{align}\label{dome1}
\displaystyle \D_1&=(1-x)D_{q}+\frac{1}{2(q-1)}(D_{b,c}-2I),\\\label{dome2}
\displaystyle \D_2&=D_{1/q}+\frac{q}{2c(q-1)}D_{b,c},\\\label{dome3}
\displaystyle \D_3&=-(x+bc)D_{q}+\frac{1}{2(q-1)}(D_{b,c}-2I) ,
\end{align}
where $D_q$ is the $q$-derivative (\ref{defdq}).
\end{lemma}

\begin{proof}
We include a detailed proof of the formula (\ref{dome2}) for the $\D$-operator $\D_2$.
The proofs for the other two $\D$-operators can be deduced in a similar way and have been omitted.

Using the definition of a $\D$-operator (\ref{defTo2}) and doing some
straightforward calculations the expression (\ref{dome2}) for the $\D$-operator $\D_2$
can be rewritten as
\begin{equation}\label{ide-D2}
\frac{q}{(1-bq)}m_{n-1}^{bq,c/q;q}(x)=
\sum_{j=1}^n (-1)^{j+1}\frac{q^{n+1-j}}{(bq^{n-j+1};q)_{j}}m_{n-j}^{b,c;q}(x).
\end{equation}
From the Verma's formula \cite[Eq. (3.7.9) page 76]{GR} the following identity holds
$$
{_3\phi_2}\left(\!\begin{array}{c} q^{-n} ,abq^{n+1}, x \\
aq,cq\end{array}; q,q \right)=\sum_{j=0}^n c_{nj}
 {_3\phi_2}\left(\begin{array}{c} q^{-n} ,\alpha\beta q^{n+1}, x \\
\alpha q,\gamma q\end{array}; q,q \right),
$$
where
$$
c_{nj}=(-1)^j q^{\binom{j+1}{2}}
\frac{(q^{-n},a b q^{n+1},\alpha q, \gamma q;q)_j}
{(q, aq, cq, \alpha\beta q^{j+1};q)_j}\, %\times\\[8mm]
{}_{4}\phi_3 \left(\!\begin{array}{c} q^{j-n},
a b q^{n+j+1},\alpha q^{j+1}, \gamma q^{j+1} \\
\alpha\beta q^{2j+2},a q^{j+1},c q^{j+1}
\end{array} ; q , q\! \right).
$$
Choosing $b=-c/(ad)$, $\beta=-\gamma/(\alpha\delta)$, and taking  limits as
$c\to\infty$ and $\gamma\to\infty$, we find
$$
{_2\phi_1}\left(\!\begin{array}{c} q^{-n} , x \\
aq \end{array}; q,-\frac{q^{n+1}}d \right)=\sum_{j=0}^n c_{nj}
{_2\phi_1}\left(\begin{array}{c} q^{-j} , x \\
\alpha q, \end{array}; q,-\frac{q^{j+1}}\delta \right),
$$
$$
c_{nj}=(-1)^j q^{\binom{j+1}{2}-j^2+nj}
\frac{(q^{-n},\alpha q ;q)_j\delta^j} {(q, aq ;q)_j d^j}
{}_{2}\phi_1 \left(\!\begin{array}{c} q^{j-n},\alpha q^{j+1} \\
a q^{j+1}\end{array} ; q , q^{n-j}\frac{\delta}{d}\! \right).
$$
If we now substitute $a=bq$, $d=c/q$, $\alpha=b$, $\delta=c$ and use the
summation formula (II.7) page 354 of \cite{GR}, as well as \eqref{defmep},
we obtain
$$
m_n^{bq,c/q;q}(x)=\sum_{j=0}^n \frac{(-1)^{n-j}q^j(1-qb)}{(bq^{j+1};q)_{n-j+1}}m_j^{b,c;q}(x),
$$
which it is equivalent to \eqref{ide-D2}.
\end{proof}

\bigskip

\begin{remark}\label{sirme}
In the next three subsections, using the $\D$-operators displayed
in Lemma \ref{lTme} we construct from the $q$-Meixner polynomials new families of $q$-Krall polynomials (that is, new families of orthogonal polynomials which are also common eigenfunctions for higher order $q$-difference operators). To do that, we proceed as follows. In the Subsection $4.i$, $i=1,2,3$, we work with the $\D$-operator $\D_i$ and the corresponding sequences $\varepsilon_{n,i}$ and $\sigma _{n,i}$ (see (\ref{veme1}), (\ref{veme2}), (\ref{veme3})).
We then  consider the
polynomials $q_n$, $n\ge 0$, defined by (\ref{defqng2})
$$
q_0(x)=1,\quad \quad q_n(x)=m_n^{b,c;q}(x)+\varepsilon_{n,i}\frac{P_2(q^n)}{P_2(q^{n-1})}m_{n-1}^{b,c;q}(x),\quad  n\ge 1,
$$
where  $P_2$ is an arbitrary polynomial (with the only assumption that $P_2(q^n)\not =0$, $n\ge 0$). Notice that the sequence of polynomials $(q_n)_n$ depends on $i$ and $P_2$. Then, since the $q$-Meixner polynomials are eigenfunctions of the second order $q$-difference operator $D_{b,c}$,
Lemma \ref{fl2v} implies that the polynomials $q_n$, $n\ge 0$, will be eigenfunctions of a higher order $q$-difference operator (explicitly given by (\ref{defD2})). However, the polynomials $(q_n)_n$ are not always orthogonal with respect to a measure. Actually, only
for a convenient choice of the polynomial $P_2$, these polynomials $(q_n)_n$ will also be orthogonal with respect to a measure. For the $q$-Meixner polynomials a very nice symmetry in the choice of  the polynomial $P_2$ appears.
Indeed, fixed a positive integer $k$, then when we choose the polynomial $P_2$ to be  a $k$-th $q$-Meixner polynomial but modifying the parameters $b,c$ and $q$ in certain convenient way (which depends on the operator $\D_i$), the polynomials $(q_n)_n$ are orthogonal as well. Moreover, in each case, the orthogonalizing measure for the polynomials $(q_n)_n$ turns out to be a Christoffel transform of the $q$-Meixner measure.
More precisely, for a fixed $k\ge 1$ and for each one of the three different operators $\D_i$  displayed in the previous Lemma,
we show in the following table our choice of the polynomial $P_2$ ($\deg (P_2)=k$) and the orthogonalizing measure $\tilde \rho_{k,b,c}^q$ for the polynomials $(q_n)_n$.

\bigskip

\begin{center}
\begin{tabular}{||l | c| c||}
\hline
$\D$-operator & $P_2(x)$ & Measure $\tilde \rho _{k,b,c}^q$ \\
\hline
$\D_1$  & $m_k^{-c,1/(bc);q}(qx)$ & $(x+bcq)\cdots (x+bcq^k)\rho_{b,q^{1+k}c}^q$ \\
\hline
$\D_2$  & $m_k^{b,c;1/q}(bx)$ & $(x-b)\cdots (x-b/q^{k-1})\rho_{b/q^{k+1},q^{1+k}c}^q$\\
\hline
$\D_3$  & $m_k^{1/b,bc;q}(qx)$ & $(x-q)\cdots (x-q^k)\rho_{b/q^{1+k},c}^q(x+k+1)$ \\
\hline
\end{tabular}
\end{center}

\bigskip
A final remark before going on with the proofs. In all the cases, the key to prove that the polynomials $(q_n)_n$ are orthogonal with respect to $\tilde\rho _{k,b,c}^q$ is the explicit computation of $\langle \tilde \rho _{k,b,c}^q,m_n^{b,c;q}\rangle$.

\end{remark}

\subsection{$q$-Meixner I}
Here we use the first $\D$-operator above (\ref{dome1}) for the $q$-Meixner polynomial.

For $k$ a positive integer, and  $b,c$ satisfying $b\not =q^{-n}$ and $c\not =0,-q^{n-k+1}$, $n\ge 1$, let $\tilde \rho_{k,b,c}^q$ be the moment functional defined by
\begin{equation}\label{tme2}
\tilde \rho_{k,b,c}^q=(x+bcq) \cdots ( x+bcq^{k}) \rho_{b,q^{k+1}c}^q,
\end{equation}
where $\rho_{b,c}^q$ is the $q$-Meixner moment functional.

A straightforward calculation gives
\begin{equation}\label{tme22}
(x+bcq^{k+1})\tilde \rho_{k,b,c}^q=(-c;q)_{k+1}\rho_{b,c}^q.
\end{equation}
Hence, for $0<q<1$, $0\le bq<1$ and $c>0$, this
moment functional can be represented by the positive  measure
$$
\tilde \rho_{k,b,c}^q=\frac{(-bcq;q)_\infty }{(-cq^{k+1};q)_\infty }\sum _{x=0}^\infty \frac{(bq;q)_xc^xq^{\binom{x}{2}}}{(q^{-x}+bcq^{k+1})(q,-bcq;q)_x}\delta _{q^{-x}}.
$$
As explained before, in order to find orthogonal polynomials with respect to $\tilde \rho_{k,b,c}^q$, we need the following Lemma.

\begin{lemma}\label{me1x} For a positive integer $k$, we have
\begin{align}\label{chx+k+2}
\langle \tilde \rho_{k,b,c}^q,m_{n}^{b,c;q}\rangle = (-1)^{n+k}(-cq;q)_k(q;q)_km_{k}^{-c,1/(bc);q}(q^{n+1}).
\end{align}
\end{lemma}

\begin{proof}
Write
\begin{align*}
\xi_n&=\langle \tilde \rho_{k,b,c}^q,m_{n}^{b,c;q}\rangle, \quad n\ge 0,\\
\zeta _{n}&=(-1)^{n+k}(-cq;q)_k(q;q)_km_{k}^{-c,1/(bc);q}(q^{n+1}),\quad n\ge 0.
\end{align*}
The three term recurrence relation (\ref{ttrrm}) for $(m_{n}^{b,c;q})_n$ and (\ref{tme22}) give that
\begin{align}\label{mecc1a}
a_0\xi_{1}+\left( b_0+bcq^{k+1}\right)\xi_{0}-(-c;q)_{k+1}&=0,\\
\label{mecc1}
a_n\xi_{n+1}+\left( b_n+bcq^{k+1}\right) \xi_{n}+c_n\xi_{n-1}&=0, \quad n\ge 1,
\end{align}
where $(a_n)_n$, $(b_n)_n$ and $(c_n)_n$ are the recurrence coefficients for the $q$-Meixner polynomials $(m_n^{b,c;q})_n$ (see (\ref{ttrrm})).

We now prove that also
\begin{align}\label{mecc2a}
a_0\zeta_{1}+\left( b_0+bcq^{k+1}\right)\zeta_{0}-(-c;q)_{k+1}&=0,\\
\label{mecc2}
a_n\zeta_{n+1}+\left( b_n+bcq^{k+1}\right) \zeta_{n}+c_n\zeta_{n-1}&=0, \quad n\ge 1.
\end{align}
Indeed, for $n\ge 0$ we have
\begin{align*}
&a_n\zeta_{n+1}+\left( b_n+bcq^{k+1}\right) \zeta_{n}+c_n\zeta_{n-1}=
(-1)^{n+1+k}(-cq;q)_k(q;q)_k
\\
&\hspace{.5cm}\times \left(a_nm_{k}^{-c,1/(bc);q}(q^{n+2})-\left( b_n+bcq^{k+1}\right)m_{k}^{-c,1/(bc);q}(q^{n+1})+c_n
m_{k}^{-c,1/(bc);q}(q^{n})\right).
\end{align*}

Then, by writing $x=q^{n+1}$ in the second order $q$-difference equation (\ref{sodeme}) for the $q$-Meixner polynomials $(m_{k}^{-c,1/(bc);q})_k$ and
using the expression (\ref{ttrrm}) for $(a_n)_n$, $(b_n)_n$ and $(c_n)_n$ we find
$$
a_nm_{k}^{-c,1/(bc);q}(q^{n+2})-\left( b_n+bcq^{k+1}\right)m_{k}^{-c,1/(bc);q}(q^{n+1})+c_n
m_{k}^{-c,1/(bc);q}(q^{n})=0.
$$
This proves (\ref{mecc2}). Equation
(\ref{mecc2a}) follows taking into account that $m_k^{-c,1/(bc);q}(1)=(-1)^k/(q;q)_k$.

Since the sequences $(\xi_n)_n$ and $(\zeta_n)_n$ satisfy the same recurrence relation, it is enough to prove
that $\xi_0=\zeta_0$.
Write $\eta_{0,c}=\langle  \rho_{b,cq}^q,1 \rangle=1$ and
$$
\eta_{k,c}=\xi_0=\langle  \tilde \rho_{k,b,c}^q,1 \rangle=\langle (x+bcq) \cdots ( x+bcq^{k})
\rho_{b,q^{k+1}c}^q,1 \rangle.
$$
Using the fact that $(x+bcq^{k+1})\rho_{b,q^{k+1}c}=(1+cq^{k})\rho_{b,q^{k}c}$, as well as
that $\langle  \rho_{b,c}^q,1 \rangle=1$, we get
\begin{align*}
\eta_{k,c}&=\langle (x+bcq) \cdots ( x+bcq^{k}) \rho_{b,q^{k+1}c}^q,1 \rangle \\&= bcq(1-q^k)
\langle (x+bcq^2) \cdots ( x+bcq^{k})\rho_{b,q^{k+1}c}^q,1 \rangle  \\ &\hspace{1.5cm}+(1+cq^k)
\langle (x+bcq^2) \cdots ( x+bcq^{k})\rho_{b,q^{k}c}^q,1 \rangle\\&=
bcq(1-q^k)\eta_{k-1,cq}+(-cq;q)_k.
\end{align*}
On the other hand, writing
$$
\tau_{k,c}=\zeta_0=(-1)^{k}(-cq;q)_k(q;q)_k m_{k}^{-c,1/(bc);q}(q),
$$
we find
\begin{align*}
\tau_{k,c}-bcq(1-q^k)\tau_{k-1,cq} =(-1)^k&(cq^2;q)_{k-1}(q;q)_k\\ &
\times [(1+cq)m_{k}^{-c,1/(bc);q}(q)+bcq \, m_{k-1}^{-cq,1/(bcq);q}(q)].
\end{align*}
Using the forward shift operator for the $q$-Meixner polynomials \cite[Eq. (14.13.6) page 490]{KLS}
the expression inside the quadratic brackets becomes
$$
(1-cq)m_{k}^{-c,1/(bc);q}(1)=(1-cq)(-1)^ k(q;q)_k^{-1}.
$$
Thus $\tau_{k,c}-bcq(1-q^k)\tau_{k-1,cq}=(-cq;q)_k$. We then conclude that
$\tau_{k,c}$ and $\eta_{k,c}$ satisfy the same recurrence
relation with the same initial condition $\tau_{0,c}=\eta_{0,c}=1$.
Therefore, by induction on $k$ it follows that
$\tau_{k,c}=\eta_{k,c}$ for all $k$ and therefore $\xi_0=\zeta_0$. This completes the proof.
\end{proof}

We now use the Lemmas \ref{fl2v} and \ref{me1x}, to construct orthogonal polynomials which are also eigenfunctions of a higher order $q$-difference operator.

\begin{theorem} \label{lm51}
For $k\ge 1$, let $b,c$ be real numbers satisfying that $b\not =q^{-n}$, $c\not =0,-q^{n-k+1}$, $n\ge 1$, and
$m_{k}^{-c,1/(bc);q}(q^{n})\not =0$, $n\ge 1$, where $m_{k}^{-c,1/(bc);q}$, $k\ge 1$, are
$q$-Meixner polynomials (see (\ref{defmep})).
We define the sequences of numbers $(\gamma _n)_{n\ge 1}$ and $(\beta _n)_{n\ge 1}$  by
\begin{align}\label{eigchk2}
\gamma_n&=m_{k}^{-c,1/(bc);q}(q^{n}),\\
\label{defbetch2}
\beta_n&=\frac{\gamma_{n+1}}{\gamma_n},\quad n\ge 1,
\end{align}
and the sequence of polynomials $(q_n)_n$ by $q_0=1$, and
\begin{equation}\label{defqch2}
q_n(x)=m_{n}^{b,c;q}(x)+\beta _nm_{n-1}^{b,c;q}(x), \quad n\ge 1.
\end{equation}
Then the polynomials $(q_n)_n$ are orthogonal with respect
to the moment functional $\tilde \rho _{k,b,c}^q$ (\ref{tme2}).
Moreover, the orthogonal polynomials $(q_n)_n$ are eigenfunctions of a $q$-difference operator of
order $2k+2$.
\end{theorem}

\begin{proof}
In the notation of Lemma \ref{addel}, we have $\lambda =-bcq^{k+1}$, $\nu=\tilde \rho _{k,b,c}^q$, $\mu =(-c;q)_{k+1}\rho_{b,c}^q$, $p_n=m_n^{b,c;q}$  and $M=0$. The orthogonality of the polynomials $(q_n)_n$ with
respect to $\tilde \rho _{k,b,c}^q$ is now an easy consequence of Lemmas \ref{addel} and \ref{me1x}, and (\ref{eigchk2}), (\ref{defbetch2}).

The second part of the Theorem is a straightforward consequence of the Lemma \ref{fl2v}.
We have just to identify who the main characters are in this example.
Indeed, write $\varepsilon_n=1$ and $\sigma_{n}=q^{n-1}/(q-1)$. The sequences $(\varepsilon_n)_n$ and $(\sigma_{n})_n$ then
define the $\D$-operator $\D _1$ in %Lemma
Proposition \ref{lTme}. In this case we have taken $P_2(x)=m_{k}^{-c,1/(bc);q}(qx)$.
Since $\theta _n=q^n$ and $\displaystyle \sigma_n=\frac{1}{q(q-1)}q^{n}$, Remark \ref{rem} guarantees the existence of a polynomial $P_1$
satisfying the hypotheses of Lemma \ref{fl2v}. Moreover, using the expression for $\D_1$ in Proposition
\ref{lTme}, the $q$-difference operator $D_{b,c}$ for the $q$-Meixner polynomials (\ref{sodeme}) and (\ref{defP1}), an explicit expression
for the $q$-difference operator for $(q_n)_n$ can be obtained. Since the polynomials $P_2$ and $P_1$ (see (\ref{defP1})) have degrees $k$ and $k+1$, respectively, a careful computation shows that the order of the $q$-difference operator for the polynomials $(q_n)_n$ is just $2k+2$.
\end{proof}

\bigskip
The cases corresponding to the operators $\D_1$ and $\D_2$ are similar so in the following two subsections we only include  sketch of the proofs.

\subsection{$q$-Meixner II}
For $k$ a positive integer, and  $b,c$ satisfying $b\not =q^{k-n+1}$ and $c\not =0,-q^{n-k+1}$, $n\ge 1$, let $\tilde \rho_{k,b,c}^q$ be the moment functional defined by
\begin{equation}\label{tme1}
\tilde \rho_{k,b,c}^q=(x-b)\left( x-\frac{b}{q}\right) \cdots \left( x-\frac{b}{q^{k-1}}\right). \rho_{b/q^{k+1},q^{k+1}c}^q
\end{equation}
A simple calculation gives
\begin{equation}\label{tme12}
\left( x-\frac{b}{q^k}\right)\tilde \rho_{k,b,c}^q= (b/q^k;q)_{k+1}(-c;q)_{k+1}\rho_{b,c}^q.
\end{equation}
Hence, for $0<q<1$, $0\le bq<1$ and $c>0$, this
moment functional can be represented by the  measure
$$
\tilde \rho_{k,b,c}^q=\frac{(b/q^k;q)_{k+1}(-bcq;q)_\infty }{(-cq^{k+1};q)_\infty }\sum _{x=0}^\infty \frac{(bq;q)_xc^xq^{\binom{x}{2}}}{(q^{-x}-b/q^k)(q,-bcq;q)_x}\delta _{q^{-x}}.
$$
The polynomials orthogonal with respect to $\tilde \rho _{k,b,c}^q$  are also eigenfunctions of a
higher order $q$-difference operator.

\begin{theorem}
For $k\ge 1$,  let $b,c$ be real numbers satisfying that $b\not =q^{k-n+1}$, $c\not =0,-q^{n-k+1}$, $n\ge 1$, and
$m_{k}^{b,c;1/q}(q^{n-1})\not =0$, $n\ge 1$, where $m_{k}^{b,c;1/q}$, $k\ge 1$, are
$q$-Meixner polynomials (see (\ref{defmep})).
We define the sequences of numbers $(\gamma _n)_{n\ge 1}$ and $(\beta _n)_{n\ge 1}$  by
\begin{align}\label{eigchk}
\gamma_n&=m_{k}^{b,c;1/q}(bq^{n-1}),\\
\label{defbetch}
\beta_n&=\frac{\gamma_{n+1}}{(1-bq^n)\gamma_n},\quad n\ge 1,
\end{align}
and the sequence of polynomials $(q_n)_n$ by $q_0=1$, and
\begin{equation}\label{defqch}
q_n(x)=m_{n}^{b,c;q}(x)+\beta _nm_{n-1}^{b,c;q}(x), \quad n\ge 1.
\end{equation}
Then the polynomials $(q_n)_n$ are orthogonal with respect
to the moment functional $\tilde \rho _{k,b,c}^q$ (\ref{tme1}).
Moreover, the orthogonal polynomials $(q_n)_n$ are eigenfunctions of a $q$-difference operator of
order $2k+2$.
\end{theorem}

\begin{proof}
The proof is similar to that of Theorem \ref{lm51}, but using
$$
\langle \tilde \rho_{k,b,c}^q,m_{n}^{b,c;q}\rangle = \frac{(-1)^nc^k(b/q^k;q)_k(q;q)_km_{k}^{b,c;1/q}(bq^n)}{(bq;q)_n}
$$
instead of Lemma \ref{me1x}, and the operator $\D_2$ (see (\ref{dome2})) instead of $\D_1$.
\end{proof}

\subsection{$q$-Meixner III}
For $k$ a positive integer, and  $b,c$ satisfying $b\not =q^{k-n+1}$ and $c\not =0,-q^n$, $n\ge 1$, let $\tilde \rho_{k,b,c}^q$ be the moment functional defined by
\begin{equation}\label{tme3}
\tilde \rho_{k,b,c}^q=(x-q) \cdots ( x-q^{k}) \rho_{b/q^{k+1},c}^q(x+k+1).
\end{equation}
A direct calculation leads to the relation
\begin{equation}\label{tme23}
(x-q^{k+1})\tilde \rho_{k,b,c}^q=c^{k+1}q^{\binom{k+1}{2}}(b/q^k;q)_{k+1}\rho_{b,c}^q.
\end{equation}
Hence, for $0<q<1$, $0\le bq<1$ and $c>0$, this
moment functional can be represented by the  measure
\begin{align*}
\tilde \rho_{k,b,c}^q&=\frac{(-1)^kq^{\binom{k+1}{2}}(q;q)_k(-bc/q^k;q)_\infty}{(-c;q)_\infty }\delta _{q^{k+1}}\\&\hspace{1cm}+\frac{q^{\binom{k+1}{2}}(-bcq;q)_\infty }{(-c;q)_\infty }\sum _{x=0}^\infty \frac{(bq^{-k};q)_{x+k+1}c^{x+k+1}q^{\binom{x}{2}}}{(q^{-x}-q^{k+1})(q,-bcq;q)_x}\delta _{q^{-x}}.
\end{align*}
The orthogonal  polynomials with respect to the above functional
$\tilde \rho _{k,b,c}^q$  are also eigenfunctions of a
higher order $q$-difference operator.

\begin{theorem}
For $k\ge 1$,  let $b,c$ be real numbers satisfying that $b\not =q^{k-n+1}$, $c\not =0,-q^n$, $n\ge 1$, and
$m_{k}^{-c,1/(bc);q}(q^{n})\not =0$, $n\ge 1$, where $m_{k}^{-c,1/(bc);q}$, $k\ge 1$, are
$q$-Meixner polynomials (see (\ref{defmep})).
We define the sequences of numbers $(\gamma _n)_{n\ge 1}$ and $(\beta _n)_{n\ge 1}$  by
\begin{align}\label{eigchk3}
\gamma_n&=m_{k}^{1/b,bc;q}(q^{n}),\\
\label{defbetch3}
\beta_n&=\frac{c+q^n}{c(1-bq^n)}\frac{\gamma_{n+1}}{\gamma_n},\quad n\ge 1,
\end{align}
and the sequence of polynomials $(q_n)_n$ by $q_0=1$, and
\begin{equation}\label{defqch3}
q_n(x)=m_{n}^{b,c;q}(x)+\beta _nm_{n-1}^{b,c;q}(x), \quad n\ge 1.
\end{equation}
Then the polynomials $(q_n)_n$ are orthogonal with respect
to the moment functional $\tilde \rho _{k,b,c}^q$ (\ref{tme3}).
Moreover, the orthogonal polynomials $(q_n)_n$ are eigenfunctions of a $q$-difference operator of
order $2k+2$.
\end{theorem}

\begin{proof}
The proof is similar to that of Theorem \ref{lm51}, but using
$$
\langle \tilde \rho_{k,b,c}^q,m_{n}^{b,c;q}\rangle = \frac{(-1)^{n+k}(q;q)_k(b/q^k;q)_kc^{k}q^{\binom{n+1}{2}+\binom{k+1}{2}}(-c/q^n;q)_n}{c^n(bq;q)_n} m_{k}^{1/b,bc;q}(q^{n+1})
$$
instead of (\ref{chx+k+2}),
and the operator $\D_3$ (see (\ref{dome3})) instead of $\D_1$.
\end{proof}

\section{$q$-Laguerre case}
In this section we will construct $q$-Krall polynomials from the $q$-Laguerre polynomials.
We start by recalling the basic definitions and facts about this family of $q$-polynomials.

For $\alpha\not =-1, -2, \cdots $, we define the sequence of $q$-Laguerre polynomials
$(L_n^{\alpha;q})_n$  by
%we write $(L_n^{\alpha;q})_n$ for the sequence of $q$-Laguerre polynomials defined by
\begin{equation}\label{deflap}
L_n^{\alpha;q}(x)=\frac{(-1)^n}{(q^{\alpha +1};q)_n(q;q)_n}\sum_{j=0}^n\frac{(q^{-n};q)_j(-x;q)_j}{(q;q)_j}q^{j(n+\alpha +1)}
\end{equation}
(we have taken a slightly different normalization from the one used in  \cite{KLS}, pp, 522-526,
from where the next formulas can be easily derived).

The $q$-Laguerre polynomials are eigenfunctions of the following second order $q$-difference operator
\begin{align}\label{sodelag}
D_{\alpha}(p)&=\frac{1}{x}p(x/q)-\frac{1+q^\alpha }{x}p(x)+\frac{q^\alpha (1+x)}{x}p(qx),\\\nonumber
 D_{\alpha} (L_{k}^{\alpha;q})&=q^{\alpha +k}L_{k}^{\alpha;q},\quad k\ge 0.
\end{align}
They are always orthogonal with respect to a moment functional $\rho_{\alpha}^q$, which we
normalize
by taking $\langle \rho_{\alpha}^q,1\rangle =1$. For $0<q<1$, $-1<\alpha$ this moment functional can be represented by several positive measures (the corresponding moment problem is indeterminate); one of them is the following:
\begin{equation}\label{lagw}
\rho_{\alpha }^q= \frac{(q;q)_\infty}{(q^{-\alpha};q)_\infty}\frac{x^\alpha}{(-x;q)_\infty},
\quad x>0.
\end{equation}

For $q$-Laguerre polynomials we have found two different $\D$-operators.

\begin{lemma}\label{lTlag}
For $\alpha\not =-1, -2, \cdots $, consider the $q$-Laguerre polynomials $(L_{n}^{\alpha ;q})_n$ (\ref{deflap}). Then, the operators $\D _i$,
$i=1,2$, defined by (\ref{defTo2}) from the sequences ($n\ge 0$)
\begin{align}\label{velag1}
\varepsilon _{n,1}&=1,\quad &\sigma_{n,1}&=q^{n-1},\\\label{velag2}
\varepsilon _{n,2}&=\frac{1}{1-q^{\alpha +n}},&\sigma_{n,2}&=q^{n-1},
\end{align}
are $\D$-operators for $(L_{n}^{\alpha ;q})_n$ and the algebra $\A _q$  (\ref{algdiff}) of $q$-difference
operators  with rational coefficients.
More precisely
\begin{align}\label{dolag1}
\displaystyle \D_1&=\frac{1}{2}\left( (q-1)(1-x)D_{q}+\frac{1-q}{q^{\alpha +1}}D_{1/q}-I\right),\\\label{dolag2}
\displaystyle \D_2&=\frac{1}{2}\left( -(q-1)(1+x)D_{q}+\frac{1-q}{q^{\alpha +1}}D_{1/q}-I\right).
\end{align}
\end{lemma}

\begin{proof}
We include a detailed proof of the formula (\ref{dolag2}) for the $\D$-operator $\D_2$.
The proof for the other  $\D$-operator can be deduced in a similar way and has been omitted.

We need the two following standard formulas for $q$-Laguerre polynomials
\begin{align}\label{mmm1}
(-1)^n(q^{\alpha +1};q)_nL_n^{\alpha ;q}(x)&=\sum_{j=0}^n(-q)^{n-j}(q^{\alpha };q)_{n-j}L_{n-j}^{\alpha -1 ;q}(x),\\\label{mmm2}
D_{1/q}(L_n^{\alpha ;q}(x))&=\frac{q^{\alpha +1}}{(1-q)(1-q^{\alpha +1})}L_{n-1}^{\alpha +1 ;q}(x).
\end{align}

Using the definition of a $\D$-operator (\ref{defTo2}), the second order $q$-difference operator for the $q$-Laguerre polynomials (\ref{sodelag}) and then (\ref{mmm1}) and (\ref{mmm2}), we have
\begin{align*}
\D_2(L_{n}^{\alpha ;q})&=-\frac{q^n}{2}L_{n}^{\alpha ;q}(x)+\sum_{j=1}^n (-1)^{j+1}\frac{q^{n-j}}{\prod_{i=0}^{j-1}(1-q^{\alpha +n-i})}L_{n-j}^{\alpha ;q}(x)=\\
&=-\frac{q^n}{2}L_{n}^{\alpha ;q}-\frac{(-1)^n}{(q^{\alpha +1};q)_n}\sum_{j=1}^n (-q)^{n-j}(q^{\alpha +1};q)_{n-j}L_{n-j}^{\alpha ;q}\\
&=-\frac{q^n}{2}L_{n}^{\alpha ;q}+\frac{1}{1-q^{\alpha+1}}L_{n-1}^{\alpha +1 ;q}(x)\\
&=\frac{-1}{2q^\alpha}D_\alpha (L_{n}^{\alpha ;q})+\frac{1-q}{q^{\alpha +1}}D_{1/q}(L_n^{\alpha ;q}(x)).
\end{align*}
This gives that $\D_2=\frac{-1}{2q^\alpha}D_\alpha +\frac{1-q}{q^{\alpha +1}}D_{1/q}$. A straightforward calculation
using (\ref{sodelag}) completes the proof.

\end{proof}

\bigskip
In the next two subsections, using the $\D$-operators displayed
in Lemma \ref{lTlag} we construct from the $q$-Laguerre polynomials new families of $q$-Krall polynomials.
The strategy is similar to what we have done for the $q$-Meixner polynomials in the previous Section (see Remark \ref{sirme}), and hence only sketch of the proofs are included. However, there are some differences, mainly for the operator $\D_2$. In this case, we have to assume that the parameter $\alpha $ is a positive integer. The choice of $P_2$ is not so symmetric as in the $q$-Meixner case (in this case $P_2$ is, essentially, a $q$-Pochhammer symbol). In particular, the orthogonalizing measure for the polynomials $(q_n)_n$ is equal to $\rho_{\alpha -1} ^q$ plus a Dirac delta at $0$. But, on the other hand, we can introduce a new parameter in the polynomials $(q_n)_n$: the mass $M$ at the Dirac delta at $0$.

The situation with the operator $\D_1$ is quite similar to the $q$-Meixner case, if we except that to get the orthogonality of $(q_n)_n$
we have to choose the polynomial $P_2$ in the relative family of Al-Salam-Carlitz polynomials $(v_n^{a;q})_n$.

For a fixed $k\ge 1$ and for each one of the two different operators $\D_1$ and $\D_2$  displayed in the previous Lemma,
we show in the following table our choice of the polynomial $P_2$ ($\deg (P_2)=k$) and the orthogonalizing measure $\tilde \rho_{k,\alpha }^q$ for the polynomials $(q_n)_n$.

\bigskip

\begin{center}
\begin{tabular}{||l | c| c||}
\hline
$\D$-operator & $P_2(x)$ & Measure $\tilde \rho _{k,\alpha }^q$ \\
\hline
$\D_1$  & $v_k^{1/q^\alpha;q}(x/q^{\alpha -1})$ & $(1+x/q)\cdots (1+x/q^k)\rho_{\alpha}^q(x/q^{n+1})$ \\
\hline
$\D_2$; $\alpha =k$  & $1+M(x/q^{k -2};q)_k/(q;q)_k$ & $\rho_{k -1}^q+M\delta_0$\\
\hline
\end{tabular}
\end{center}

\bigskip

\subsection{$q$-Laguerre I}
For $k$ a positive integer, and  $\alpha\not =-1, -2, \cdots $, let $\tilde \rho_{k,\alpha }^q$ be the moment functional defined by
\begin{equation}\label{tlag1}
\tilde \rho_{k,\alpha}^q=\left(1 +\frac{x}{q}\right) \cdots \left(1 +\frac{x}{q^k}\right) \rho_{\alpha}^q(x/q^{k+1}).
\end{equation}
A simple calculation gives that
\begin{equation}\label{tlag12}
\left(1 +\frac{x}{q^{k+1}}\right)\tilde \rho_{k,\alpha}^q=\frac {1}{q^{\alpha (k+1)}}\rho_{\alpha }^q.
\end{equation}
Hence, for $0<q<1$, $-1<\alpha$, this moment functional can be represented by the positive measure
$$
\tilde \rho_{k,\alpha}^q=\frac{(q;q)_\infty}{q^{\alpha (k+1)}(q^{-\alpha};q)_\infty}\frac{x^\alpha}{(1+x/q^{k+1})(-x;q)_\infty},\quad x>0.
$$
We also need the Al-Salam-Carlitz polynomials $(v_n^{a,q})_n$ (\cite{KLS}, pp. 537--540) defined by
\begin{equation}\label{defasc}
v_n^{a;q}(x)=\sum_{j=0}^n\frac{(-1)^j(q^{-n};q)_j(x;q)_j q^{-\binom{2}{2}+jn}}{a^j(q;q)_j}.
\end{equation}
Orthogonal polynomials with respect to $\tilde \rho _{k,\alpha}^q$  are also eigenfunctions of a
higher order $q$-difference operator.

\begin{theorem} \label{lm533}
For $k\ge 1$, let $\alpha$ be a real number satisfying that $\alpha\not =-1, -2, \cdots $, and $v_{k}^{1/q^\alpha;q}(q^{n})\not =0$, $n\ge 1$, where $v_{k}^{1/q^\alpha ;q}$, $k\ge 1$, are
the Al-Salam-Carlitz polynomials (see (\ref{defasc})).
We define the sequences of numbers $(\gamma _n)_{n\ge 1}$ and $(\beta _n)_{n\ge 1}$  by
\begin{align}\label{eiglagk}
\gamma_n&=v_{k}^{1/q^\alpha;q}(q^{n}),\\
\label{defbetlag}
\beta_n&=\frac{\gamma_{n+1}}{\gamma_n},\quad n\ge 1,
\end{align}
and the sequence of polynomials $(q_n)_n$ by $q_0=1$, and
\begin{equation}\label{defqlag}
q_n(x)=L_{n}^{\alpha ;q}(x)+\beta _nL_{n-1}^{\alpha;q}(x), \quad n\ge 1.
\end{equation}
Then the polynomials $(q_n)_n$ are orthogonal with respect
to the moment functional $\tilde \rho _{k,\alpha}^q$ (\ref{tlag1}).
Moreover, the orthogonal polynomials $(q_n)_n$ are eigenfunctions of a $q$-difference operator of order $2k+2$.
\end{theorem}

\begin{proof}
The proof is similar to that of Theorem \ref{lm51}, but using
$$
\langle \tilde \rho_{k,\alpha}^q,L_{n}^{\alpha;q}\rangle = (-1)^{n+k}q^{\binom{k}{2}}v_{k}^{1/q^\alpha;q}(q^{n+1})
$$
instead of Lemma \ref{me1x}, and the operator $\D_1$ (see (\ref{dolag1})) instead of the $\D$-operator defined by (\ref{dome1}).
\end{proof}

\subsection{$q$-Laguerre II}
For $\alpha \not =0,-1,-2, \cdots $ and $M\not =-(q;q)_n/(q^{\alpha+1};q)_n$, $n\ge 1$, consider
the moment functional $\tilde \rho_{\alpha, M}^q$ defined by
\begin{equation}\label{darlag}
\tilde \rho_{M, \alpha}^q=M\delta _0+\rho _{\alpha-1}^q,
\end{equation}
where $\rho_\alpha ^q$ is the orthogonalizing moment functional for the $q$-Laguerre polynomials $(L_{n}^{\alpha ;q})_n$.

Hence, for $0<q<1$, $0<\alpha$ and $M\ge 0$, this moment functional can be represented by the positive measure
\begin{equation}\label{darlag2}
\tilde \rho_{M,\alpha}^q=M\delta _0+\frac{(q;q)_\infty}{(q^{-\alpha+1};q)_\infty}\frac{x^{\alpha -1}}{(-x;q)_\infty},\quad x>0.
\end{equation}

Using Lemma \ref{addel}, we
 find explicitly a sequence of orthogonal polynomials with respect to $\tilde \rho _{M,\alpha}^q$
in terms of $p_{n}=L_{n}^{q,\alpha }$, $n\ge 0$.

\begin{lemma}\label{l4.1}
Let $(\gamma _{n})_{n\ge 1}$ be the sequence of numbers defined by
\begin{equation}\label{defdellagnn}
\gamma_{n}=1+M\frac{(q^{\alpha+1};q)_n}{(q;q)_n}.
\end{equation}
Write
\begin{equation}\label{defbetlagnn}
\beta_n=-\frac{\gamma_{n+1}}{(1-q^{\alpha +n})\gamma_{n}},\quad n\ge 1.
\end{equation}
Then the polynomials defined by $q_0=1$ and
\begin{equation}\label{defqnlagnn}
q_{n}(x)=L_{n}^{q,\alpha }(x)+\beta _nL_{n-1}^{q,\alpha }(x), \quad n\ge 1,
\end{equation}
are orthogonal with respect to $\tilde \rho_{M,\alpha}^q$ (\ref{darlag}).
\end{lemma}

\begin{proof}
In the notation of Lemma \ref{addel}, we have $\lambda=0$, $\nu =\rho _{\alpha -1}^q$, $\mu=\rho_{\alpha}^q$ and $p_n=L_{n}^{\alpha ;q}$.
The identity (\ref{mmm1})
gives $\alpha_{n}=(-1)^n/(q^{\alpha +1};q)_n$. Substituting the obtained expression for $\alpha_n$
as well as the value $p_n(0)=(-1)^n/(q;q)_n$ in (\ref{hk2})
and doing some straightforward calculations we get (\ref{defbetlagnn}).
%Since we have also $p_n(0)=(-1)^n/(q;q)_n$,
%an straightforward computation gives  (\ref{defbetlagnn}) from (\ref{hk2}).
\end{proof}

\bigskip

When $\alpha$ is a positive integer, the sequence $(\gamma_n)_n$ (\ref{defdellagnn}) is a polynomial
in $q^{n}$ of degree $\alpha$, and then Lemma \ref{fl2v} and Remark \ref{rem} give a  $(2\alpha +2)$-th order $q$-difference operators for which the  polynomials (\ref{defqnlagnn}) are eigenfunctions.

\begin{theorem} \label{lm53}
Let $\alpha $ be a positive integer.
Consider the sequences of numbers $(\gamma _n)_{n\ge 1}$ and $(\beta _n)_{n\ge 1}$ defined  by (\ref{defdellagnn})
and (\ref{defbetlagnn}), respectively, and  the sequence of polynomials $(q_n)_n$ defined by (\ref{defqnlagnn}).
Then the polynomials $(q_n)_n$ are orthogonal with respect
to the moment functional $\tilde \rho _{M,\alpha}^q$ (\ref{darlag2}) and are eigenfunctions
of a $q$-difference operator of
order $2\alpha +2$.
\end{theorem}

\bigskip

To conclude this paper let us point out that the above Theorems  \ref{lm533} and \ref{lm53}
suggest the following conjectures.

\bigskip

\noindent
\textbf{Conjecture B1.} Let $F$ be a finite set of positive integers, then the orthogonal polynomials with respect to the positive measure
$$
\prod_{f\in F}(1+xq^f)\rho_{\alpha}^q
$$
are eigenfunctions of a $q$-difference operator of order $2\sum_{f\in F}f-n_F(n_F-1)+2$, where $n_F$ is the number of elements of $F$.

\bigskip
\noindent
\textbf{Conjecture B2.} Given a nonnegative integer $K$, since the moment problem associated to the $q$-Laguerre polynomials is indeterminate, we have for real numbers $M_j$, $j=0,\cdots , K$, with small enough absolute value, that the moment functional
$$
\prod_{f\in F}(1+xq^f)\rho_{\alpha}^q+\sum_{j=0}^KM_j\delta ^{(j)}_0
$$
can be represented by a positive measure.
If $\alpha $ is a nonnegative integer with $\alpha >K$ then, the orthogonal polynomials with respect to that positive measure are eigenfunctions of a higher order $q$-difference operator.

     \end{document}